\documentclass[12pt]{amsart}
\usepackage{amssymb}

\theoremstyle{plain}
\newtheorem{theorem}{Theorem}[section]
\newtheorem{corollary}[theorem]{Corollary}
\newtheorem{lemma}[theorem]{Lemma}
\newtheorem{proposition}[theorem]{Proposition}

\theoremstyle{definition}
\newtheorem{definition}[theorem]{Definition}

\theoremstyle{remark}
\newtheorem*{example}{Example}

\def\span{{\rm span}\,}
\def\codim{{\rm codim}\,}

\newcommand{\abs}[1]{\lvert#1\rvert}
\newcommand{\norm}[1]{\lVert#1\rVert}
\newcommand{\bigabs}[1]{\bigl\lvert#1\bigr\rvert}
\newcommand{\bignorm}[1]{\bigl\lVert#1\bigr\rVert}

\newcommand{\Bignorm}[1]{\Bigl\lVert#1\Bigr\rVert}
\renewcommand{\le}{\leqslant}
\renewcommand{\ge}{\geqslant}

\newcommand{\term}[1]{{\textit{\textbf{#1}}}}

\begin{document}
\baselineskip 18pt

\title[Almost invariant half-spaces of algebras of operators]
{Almost invariant half-spaces of\\ algebras of operators}

\author[A.I.~Popov]{Alexey I. Popov}
\address{Department of Mathematical and Statistical Sciences, University of Alberta, Edmonton, AB, T6G\,2G1. Canada}
\email{apopov@math.ualberta.ca}

\date{\today.}
\begin{abstract} Given a Banach space $X$ and a bounded linear operator $T$ on $X$, a subspace $Y$ of $X$ is almost invariant under $T$ if $TY\subseteq Y+F$ for some finite-dimensional ``error'' $F$. In this paper, we study subspaces that are almost invariant under every operator in an algebra $\mathfrak A$ of operators acting on $X$. We show that if $\mathfrak A$ is norm closed then the dimensions of ``errors'' corresponding to operators in $\mathfrak A$ must be uniformly bounded. Also, if $\mathfrak A$ is generated by a finite number of commuting operators and has an almost invariant half-space (that is, a subspace with both infinite dimension and infinite codimension) then $\mathfrak A$ has an invariant half-space.
\end{abstract}

\maketitle

\section{Introduction}

The notion of an almost invariant subspace was recently introduced in~\cite{APTT}. If $T$ is an operator on a Banach space $X$ then a subspace $Y$ of $X$ is called \term{almost invariant under $T$} if there exists a finite-dimensional subspace $F$ of $X$ such that 
\begin{equation}\label{Y+F}
TY\subseteq Y+F.
\end{equation}
Clearly, if $Y$ has finite dimension or finite codimension then $Y$ is almost invariant under every operator on $X$.

\begin{definition}\cite{APTT} A subspace $Y\subseteq X$ is called a \term{half-space} if $Y$ is both of infinite dimension and of infinite codimension.
\end{definition}

The question whether every operator on a Banach space has an almost invariant half-space was posed in~\cite{APTT}; it was solved there for certain classes of operators. Just as the studies of transitive algebras generalize the Invariant Subspace Problem for a single operator, the purpose of this paper is to introduce and study the notion of a subspace that is simultaneously \emph{almost} invariant under every operator in a given algebra of operators.

Throughout the paper, $X$ is a Banach space. The term ``subspace'' refers to a norm closed linear subspace of $X$, while the term ``linear subspace'' refers to a subspace that is not necessarily closed. Whenever we say that $\mathfrak A$ is an algebra of operators, we mean that $\mathfrak A$ is an algebra of operators on $X$. Also, given a sequence $(x_i)$, we write $[x_i]$ for the closed linear span of $(x_i)$.

\begin{definition}
Let $\mathcal C\subseteq L(X)$ be an arbitrary collection of operators and $Y\subseteq X$ a subspace of $X$. We call  $Y$ \term{almost invariant under $\mathcal C$}, or \term{$\mathcal C$-almost invariant} if $Y$ is almost invariant under every operator in~$\mathcal C$.
\end{definition}

Like in the case of a single operator, every subspace that is not a half-space is automatically almost invariant under every collection $\mathcal C$ of operators on~$X$.

In Section~\ref{errors-study}, we study the finite-dimensional ``errors'' $F$ appearing in formula~\eqref{Y+F} corresponding to operators in an algebra $\mathfrak A$. We prove that if $\mathfrak A$ is an algebra without invariant half-spaces then for an $\mathfrak A$-almost invariant half-space $Y$ these finite-dimensional subspaces cannot be the same (Proposition~\ref{same-F}). On the other hand, we prove (Theorem~\ref{uniform-bdd}) that if $\mathfrak A$ is norm closed then these finite-dimensional subspaces cannot be ``too far apart''; the dimensions of these subspaces must be uniformly bounded.

In Section~\ref{inv-study}, the invariant subspaces of algebras having almost invariant half-spaces are investigated. It is proved that if $\mathfrak A$ is
a norm closed algebra generated by a single operator then existence of an $\mathfrak A$-almost invariant half-space implies existence of an $\mathfrak A$-invariant half-space (Theorem~\ref{inv-existence}). This theorem then is generalized to the case of a commutative algebra generated by a finite number of operators (Theorem~\ref{inv-existence-ext}). Also, the question of whether the almost invariant half-spaces of $\mathfrak A$ and $\overline{\mathfrak A}^{WOT}$ are the same is investigated (Corollary~\ref{A-iff-WOT-A} and an example after Theorem~\ref{inv-existence}).

In this paper we will occasionally refer to some standard facts about invariant subspaces of operators and algebras of operators. For a general account on invariant subspaces and transitive algebras, see~\cite{RR03}. A good review of this topic can be found in~\cite{AA02}.

{\bf Acknowledgement.} The author wants to express his deep gratitude to V.~Troitsky for many useful discussions and suggestions and wishes to thank A.~Tcaciuc for useful comments.

\section{Finite-dimensional ``errors'' of almost invariant half-spaces}\label{errors-study}

Observe that the finite-dimensional subspace $F$ appearing in the equation~\eqref{Y+F} is by no means unique. However the minimal dimension of a subspace satisfying this condition is unique. Some simple properties of a subspace of minimal dimension satisfying~\eqref{Y+F} are collected in the following lemma.

\begin{lemma}\label{char-min-dim}
Let $Y\subseteq X$ be a subspace, $\mathcal C$ be a collection of bounded 
operators on $X$ and $G\subseteq X$ be a finite-dimensional space of the 
smallest dimension such that $TY\subseteq Y+G$ for all $T\in\mathcal C$. Then
\begin{enumerate}
	\item $Y+G=Y\oplus G$;
	\item if $P:Y\oplus G\to G$ is the projection along $Y$ then 
	      $$\span\bigcup\limits_{T\in\mathcal C}PT(Y)=G;$$
	\item if $\mathcal C$ consists of a single operator, that is $\mathcal C=\{T\}$, and if $P:Y\oplus G\to G$ is the projection along $Y$ then 
	      $$PT(Y)=G.$$
Moreover, in this case $G$ can be chosen so that $G\subseteq TY$.	      
\end{enumerate}
\end{lemma}
\begin{proof}
(i) Suppose that there exists a non-zero $g\in G\cap Y$. Build $g_2,\dots 
g_n\in G$ such that $\{g,g_2,\dots, g_n\}$ is a basis of $G$. Denote 
$G_1 = \span\{g_2,\dots,g_n\}$. It is clear that $TY\subseteq Y+G_1$ for all 
$T\in\mathcal C$. However $\dim G_1<\dim G$.

(ii) Define $F=\span\{g\in G\colon v+g\in TY\mbox{ for some }v\in Y, T\in\mathcal C\}$. Clearly $F=\span\bigcup\limits_{T\in\mathcal C}PT(Y)$.

We claim that $TY\subseteq Y+F$ for all $T\in\mathcal C$. Indeed, if 
$y\in Y$ and $T\in\mathcal C$ then $Ty=v+g$ for some $v\in Y$ and $g\in G$. By 
definition of $F$ we get: $g\in F$, hence $Ty\in Y+F$. 

Since $F\subseteq G$ and $G$ has the smallest dimension among the spaces with 
the property $TY\subseteq Y+G$ for all $T\in\mathcal C$, we get $G=F$.

(iii) The first part of this statement follows immediately from (ii). Let's prove the ``moreover'' part. Let $g_1,\dots,g_n$ be a basis of $G$. By (ii), there exist $u_1,\dots,u_n$ and $y_1,\dots,y_n$ in $Y$ such that $Tu_i=y_i+g_i$ ($i=1,\dots,n$). Put $f_i=Tu_i$ and $F=[f_i]_{i=1}^n$. Then clearly $F\subseteq TY$. Also $Y+F=Y+G$, so that $TY\subseteq Y+F$. From the minimality of $G$ we obtain that $\dim F=\dim G$.
\end{proof}

The following example shows that $\bigcup\limits_{T\in\mathcal C}PT(Y)$ may not be a linear space even in the case when $\mathcal C$ is an algebra of operators. 

\begin{example}
Let $X=\ell_2(\mathbb Z)$. Define $T,S\in L(X)$ by 
$$
Te_{0}=e_1,\quad Te_{-1}=e_2,\quad Te_i=0\mbox{ if }i\ne 0,-1,
$$
and
$$
Se_0=e_3,\quad Se_i=0\mbox{ if }i\ne 0.
$$
Since $T^2=S^2=TS=ST=0$, the algebra $\mathfrak A$ generated by $T$ and $S$ consists exactly of the operators of form $aT+bS$ where $a$ and $b$ are arbitrary scalars.

Let $Y=[e_i]_{i\le 0}$. Then clearly $\mathfrak AY\subseteq Y+F$ where $F=\span\{e_1,e_2,e_3\}$, and $F$ is the space of the smallest dimension satisfying this condition. If $P:Y\oplus F\to F$ is the projection along $Y$ then $\bigcup_{R\in\mathfrak A}PR(Y)$ is not a linear space. If it were, it would have been equal to $F$, since it contains the basis of $F$. However the vector $e_2+e_3$ is not in this union.
\end{example}

Suppose $Y$ is a half-space that is almost invariant under a collection $\mathcal C$ of operators on $X$, that is, formula~\eqref{Y+F} holds for every operator $T$ in $\mathcal C$ with some $F$. One may ask if it is possible that $F$ does not depend on $T$. The following simple reasoning shows that in case of algebras of operators, this can only happen if the algebra already has a common invariant half-space.

\begin{proposition}\label{same-F}
 Let $Y\subseteq X$ be a half-space and $\mathfrak A$ an algebra of 
 operators. Suppose that there exists a finite-dimensional space $F$ such that 
 for each $T\in\mathfrak A$ we have $TY\subseteq Y+F$. Then there exists a
 half-space that is invariant under $\mathfrak A$.
\end{proposition}
\begin{proof}
Let $G$ be a space of the smallest dimension such that $TY\subseteq Y+G$ for all $T\in\mathfrak A$. We claim that $Y+G$ is invariant under every operator in $\mathfrak A$.

Denote $\mathfrak A(Y)=\bigcup_{T\in\mathfrak A}TY$. Clearly $\mathfrak A(Y)$ is invariant under $\mathfrak A$. Hence, so is $\span\mathfrak A(Y)$. Denote $Z=Y+\span\mathfrak A(Y)$. Since $TY\subseteq\span\mathfrak A(Y)$ for every $T\in\mathfrak A$, we obtain that $Z$ is invariant under $\mathfrak A$.

By Lemma~\ref{char-min-dim}(ii), if $P:Y\oplus G\to G$ is a projection along $Y$ then $P\big(\span\mathfrak A(Y)\big)=G$. Hence $Y\oplus G=Y\oplus P\big(\span\mathfrak A(Y)\big)=Y+\span\mathfrak A(Y)=Z$, so that $Y\oplus G$ is invariant under $\mathfrak A$.
\end{proof}

\begin{definition}
 Let $T\in L(X)$ be an arbitrary operator and $Y\subseteq X$ be a linear subspace. We will write $d_{Y,T}$ for the smallest $n$ such that there exists $F$ with $TY\subseteq Y+F$ and $\dim F=n$.
\end{definition}

The following observation is obvious.

\begin{lemma}\label{quotient}
Let $T\in L(X)$ be an operator and $Y\subseteq X$ be a subspace. Let $q:X\to X/Y$ be a quotient map. Then $Y$ is $T$-almost invariant if and only if $(qT)|_Y$ is of finite rank. Moreover, $\dim(qT)(Y)=d_{Y,T}$.
\end{lemma}

To proceed, we need the following two auxiliary lemmas.

\begin{lemma}\label{small-indep}
 Let $Y\subseteq X$ be a linear subspace and $\{u_i\}_{i=1}^N$ be a collection 
 of linearly independent vectors in $X$ such that $[u_i]_{i=1}^N\cap Y=\{0\}$. Let $\{v_i\}_{i=1}^N\subseteq X$ be arbitrary. Then for all but finitely many $\alpha$ we have $\{v_i+\alpha u_i\}_{i=1}^N$ is linearly independent and $[v_i+\alpha u_i]_{i=1}^N\cap Y=\{0\}$.
\end{lemma}
\begin{proof}
 Let $F=\span\{u_i, v_i\colon i=1,\dots,N\}$. Let $G=\big(Y+F\big)/Y$. Denote
 $x_i=u_i+Y\in G$, $z_i=v_i+Y\in G$. Then the set $\{x_i\}_{i=1}^N$ is 
 linearly 
 independent. Clearly, to establish the lemma it is enough to prove that 
 the set $\{z_i+\alpha x_i\}_{i=1}^N$ is linearly independent for all but finitely many $\alpha$.
 
 Denote $M=\dim G$. Let $\{b_i\}_{i=1}^M$ be a basis of $G$ such that 
 $b_i=x_i$ 
 for all $1\le i\le N$.
 Denote the coordinates of vectors $z_i$ in this basis by $z_{ij}$. 
 Let $A$ be the $M\times M$-matrix with first $N$
 rows consisting of the coordinates of $z_i$ ($i=1,\dots,N$), the last $M-N$
 rows being zero rows:
 \begin{displaymath}
  A=\left[\begin{array}{ccccc}
           z_{11} & z_{12} & \cdots & z_{1,M-1} & z_{1M}\\
           \vdots &        & \cdots & & \vdots\\
           z_{N1} & z_{N2} & \cdots & z_{N,M-1} & z_{NM}\\
                0 &      0 & \cdots & 0 & 0\\
           \vdots &        & \cdots & & \vdots\\
                0 &      0 & \cdots & 0 & 0
          \end{array}
    \right]
 \end{displaymath}
 Since the spectrum of $A$ is finite, $\det(A+\alpha I)\ne 0$ for all but finitely many $\alpha$. For these $\alpha$, the
 rows of $A+\alpha I$ must be linearly independent. In particular, the first
 $N$ rows are linearly independent. However the first $N$ rows are exactly the 
 representations of the vectors $z_i+\alpha x_i$ in the basis 
 $\{b_i\}_{i=1}^M$.
\end{proof}

\begin{lemma}\label{min-dim}
Let $Y\subseteq X$ be a linear subspace and $T\in B(X)$. Let 
$f_1,\dots,f_n\in TY$ be such that no non-trivial linear combination of 
$\{f_1,\dots,f_n\}$ belongs to $Y$. Then $n\le d_{Y,T}$.
\end{lemma}
\begin{proof}
Let $q:X\to X/Y$ be the quotient map. Then $qf_1,\dots,qf_n$ are linearly independent. Since $qf_1,\dots,qf_n\in (qT)(Y)$, we get $n\le\dim(qT)(Y)=d_{Y,T}$ by Lemma~\ref{quotient}.
\end{proof}

The following theorem is the main statement in this section. Recall that, according to our convention, the term ``subspace'' stands for a norm closed subspace.

\begin{theorem}\label{uniform-bdd}
Let $\mathfrak S$ be a subspace of $L(X)$. Suppose that $Y$ is a linear subspace of $X$ that is almost invariant under~$\mathfrak S$. Then
\begin{displaymath}
 \sup\limits_{S\in\mathfrak S}d_{Y,S}<\infty.
\end{displaymath}
\end{theorem}

\begin{proof}
For every $S\in\mathfrak S$, fix a subspace $F_S\subseteq X$ such that 
$SY\subseteq Y+F_S$ and $\dim F_S=d_{Y,S}$. By Lemma~\ref{char-min-dim}, $Y+F_S$ is a direct sum. Fix $P_S:Y\oplus F_S\to F_S$ the projection along $Y$. Also fix a basis $(f_i^S)_{i=1}^{d_{Y,S}}$ of $F_S$ and a tuple $(g_i^S)_{i=1}^{d_{Y,S}}$ in $Y$ such that $(P_SS)g_i^S=f_i^S$ (this can be done by Lemma~\ref{char-min-dim}(iii)).

Suppose that the statement of the theorem is not true. Then there exists a sequence of operators $(S_k)\subseteq\mathfrak S$ such that the sequence $(d_{Y,S_k})_{k=1}^\infty$ is strictly increasing. Without loss of generality, $\norm{S_k}=1$.

We will inductively construct a sequence $(a_k)$ of scalars such that the
following two conditions are satisfied for every $m$.
\begin{enumerate}
 \item\label{cond-1} If $T_m=\sum\limits_{k=1}^m a_kS_k$ then 
       $N_m:=d_{Y,T_m}\ge d_{Y,S_m}$.
 \item\label{cond-2} Let 
$$C_m=\sup\limits_{b_1,\dots,b_{N_m}\in [-1,1]}\Bignorm{\sum_{i=1}^{N_m} b_ig_i^{T_m}}\cdot\max_{i=1,\dots,N_m}\bignorm{(f_i^{T_m})^*},$$
where $(f_i^{T_m})^*$ is the $i$-th biorthogonal functional for $(f_i^{T_m})_{i=1}^{N_m}$ in $F_{T_m}^*$, and
$$
 D_m=\min\Big\{1,\frac{1}{C_1\cdot\norm{P_{T_1}}},\dots,
 \frac{1}{C_m\cdot\norm{P_{T_m}}}\Big\}.
$$
Then $0<a_1\le\frac{1}{2}$ and $0<a_{m+1}<\tfrac{1}{2^{m+1}}D_m$ for all $m\ge 1$.
\end{enumerate}

Indeed, on the first step put $a_1=\frac{1}{2}$. Suppose that $a_1,\dots,a_m$ have been constructed. Define $D_m$ as in~\eqref{cond-2}. Denote for convenience $N=d_{Y,S_{m+1}}$. Let $u_i=f_i^{S_{m+1}}$ and $v_i=T_mg_i^{S_{m+1}}$, $i=1,\dots,N$. By Lemma~\ref{small-indep} we can find $0<\alpha<\frac{1}{2^{m+1}}D_m$ such that no non-trivial linear combination of vectors from the set $\{v_i+\alpha u_i\}_{i=1}^{N}$ is contained in $Y$. Put $a_{m+1}=\alpha$. This makes both conditions~\eqref{cond-1} and~\eqref{cond-2} satisfied for $m+1$. Indeed, condition~\eqref{cond-2} is satisfied immediately. Let's check condition~\eqref{cond-1}. Denote for convenience $y_i=g_i^{S_{m+1}}$. Observe: for each $i=1,\dots, N$, we have $T_{m+1}y_i=T_my_i+\alpha S_{m+1}y_i=v_i+\alpha u_i+w_i$ where $w_i$ is some vector in $Y$. Since no linear combination of $\{v_i+\alpha u_i\}_{i=1}^{N}$ is contained in $Y$, the same is true for $\{T_{m+1}y_i\}_{i=1}^N$. Condition~\eqref{cond-1} now follows from Lemma~\ref{min-dim}.

Denote $S=\sum\limits_{k=1}^\infty a_kS_k$. By condition~\eqref{cond-2},
$a_k\le\tfrac{1}{2^k}$ for all $k\in\mathbb N$, so that $S$ is well-defined.
For every $m\in\mathbb N$, denote $R_m=\sum\limits_{k=m+1}^\infty a_kS_k$, so
that $S=T_m+R_m$. By condition~\eqref{cond-2}, we get:
$\norm{R_m}<\frac{1}{C_m\cdot\norm{P_m}}$ for all $m\in\mathbb N$.

Clearly, $S\in\mathfrak S$. By assumptions of the theorem, $SY=Y\oplus F_S$.
Denote $n=\dim F_S<\infty$. Pick $m\in\mathbb N$ such that $N_m>n$ and put $z_i=Sg_i^{T_m}$, $i=1,\dots,N_m$. Since $N_m>n$, there exists a sequence
$(b_i)_{i=1}^{N_m}$ of scalars such that $\max_i\abs{b_i}=1$ and 
$$
z:=\sum\limits_{i=1}^{N_m}b_iz_i\in Y.
$$
Consider $y=\sum\limits_{i=1}^{N_m}b_ig_i^{T_m}$. We have 
$$
T_my=Sy-R_my=z-R_my,
$$
hence 
$$
(P_{T_m}T_m)y=-(P_{T_m}R_m)y.
$$
Clearly, for each $i=1,\dots,N_m$, we have $b_i=(f_i^{T_m})^*(P_{T_m}T_my)$.
Let $k$ be such that $\abs{b_k}=1$.
Then
\begin{equation*}
\begin{split}
1=\abs{b_k}=\bigabs{(f_k^{T_m})^*(P_{T_m}T_m\,y)}
&\ \le\bignorm{(f_k^{T_m})^*}\cdot\norm{P_{T_m}T_m\,y}=\\
=\bignorm{(f_k^{T_m})^*}\cdot\norm{P_{T_m}R_m\,y}
\le&\ \bignorm{(f_k^{T_m})^*}\cdot\norm{P_{T_m}} \cdot\norm{R_m}\cdot\Bignorm{\sum\limits_{i=1}^{N_m}b_ig_i^{T_m}}\le\\
\le\norm{P_{T_m}}\cdot\norm{R_m}\cdot&\ C_m
 <\norm{P_{T_m}}\,\frac{1}{C_m\cdot\norm{P_{T_m}}}\,C_m
=1
\end{split}
\end{equation*}
which is a contradiction.
\end{proof}

\section{Invariant subspaces of algebras having almost invariant half-spaces}
\label{inv-study}

In this section we study some connections between the invariant subspaces of an algebra of operators and the almost invariant half-spaces of this algebra. In particular, we establish that if a norm-closed algebra generated by a single operator has an almost invariant half-space then it has an invariant half-space. Then we generalize this to commutative algebras generated by a finite number of operators. Also, we study the question when the almost invariant half-spaces  of an algebra and of its WOT-closure are the same.

It is well-known that the invariant subspaces of an algebra of operators coincide with those of the WOT-closure of this algebra. Remarkably, the same statement holds for almost invariant half-spaces, provided that the algebra is norm closed.

\begin{proposition}\label{WOT-closure}
Let $Y$ be a subspace of $X$ and $\mathfrak A$ an algebra of operators acting on $X$. Let $N\in\mathbb N$ be such that $d_{Y,T}\le N$ for all $T\in\mathfrak A$. Then $d_{Y,T}\le N$ for all $T\in\overline{\mathfrak A}^{WOT}$.
\end{proposition}

\begin{proof}
It is enough to prove that if $T\in\overline{\mathfrak A}^{SOT}$ then $d_{Y,T}\le N$. Suppose this is not true. Let $T\in\overline{\mathfrak A}^{SOT}$ be an operator with $d_{Y,T}\ge N+1$. Let $F\subseteq X$ be such that $\dim F=d_{Y,T}$ and $TY\subseteq Y\oplus F$. Fix $N+1$ linearly independent vectors $(f_i)_{i=1}^{N+1}$ in $F$. By Lemma~\ref{char-min-dim}(iii), there exist $(u_i)_{i=1}^{N+1}\subseteq Y$ such that for each $i=1,\dots, N+1$ we have $Tu_i=y_i+f_i$ for some $y_i\in Y$. Since $Y\cap F=\{0\}$, $[Tu_i]_{i=1}^{N+1}\cap Y=\{0\}$ and $Tu_1,\dots,Tu_{N+1}$ are linearly independent.

Fix a net $(T_\alpha)\subseteq\mathfrak A$ such that $T_\alpha\overset{SOT}{\to} T$. Let $q:X\to X/Y$ be the quotient map. Since $[Tu_i]_{i=1}^{N+1}\cap Y=\{0\}$ and $Tu_1,\dots,Tu_{N+1}$ are linearly independent, the collection $\{(qT)u_i\}_{i=1}^{N+1}$ is linearly independent. Observe that if $\varepsilon>0$ is sufficiently small then each collection $\{v_i\}_{i=1}^{N+1}$ satisfying $\norm{v_i-(qT)u_i}<\varepsilon$ as $i=1,\dots,N+1$ is again linearly independent.

Fix $\alpha_0$ such that for all $\alpha\ge\alpha_0$ we have $\norm{(T_\alpha-T)(u_i)}<\varepsilon$ for all $i=1,\dots,N+1$. Then $\norm{(qT_\alpha-qT)(u_i)}<\varepsilon$, so that the collection $\{(qT_\alpha)u_i\}_{i=1}^{N+1}$ is linearly independent for all $\alpha\ge\alpha_0$. By Lemma~\ref{quotient}, this, however, implies that $d_{Y,T_\alpha}\ge N+1$ for all $\alpha\ge\alpha_0$ which contradicts the assumptions.
\end{proof}

\begin{corollary}\label{A-iff-WOT-A}
Let $\mathfrak A$ be a norm closed algebra of operators on $X$ and $Y$ be a half-space in $X$. Then $Y$ is $\mathfrak A$-almost invariant if and only if $Y$ is $\overline{\mathfrak A}^{WOT}$-almost invariant.
\end{corollary}

\begin{proof}
If $Y$ is $\mathfrak A$-almost invariant then by Theorem~\ref{uniform-bdd} there exists $N\in\mathbb N$ such that $d_{Y,S}<N$ for all $S\in\mathfrak A$. By Proposition~\ref{WOT-closure} the same is true for all $S\in\overline{\mathfrak A}^{WOT}$. This implies that $Y$ is $\overline{\mathfrak A}^{WOT}$-almost invariant. The converse statement is evident.
\end{proof}

We will show later in this section that the condition of $\mathfrak A$ being norm closed is essential here.

The following lemma is standard:

\begin{lemma}\label{dim-codim}
Let $X$ and $Y$ be Banach spaces and $T\in L(X,Y)$ be of finite rank. Then $\dim({\rm Range}\ T)=\codim(\ker T)$.
\end{lemma}

We will now introduce some notations. If $Y$ and $Z$ are two subspaces of $X$ and $Y\subseteq Z$ then the symbol $\codim_ZY$ will stand for the codimension of $Y$ in $Z$.

Let $T\in L(X)$ be an operator and $Y\subseteq X$ be a half-space. Consider two procedures of constructing new linear spaces:
$$
\begin{array}{ll}
D_T(Y)=\{y\in Y\colon Ty\in Y\}\quad & \mbox{``going downwards''},\\
U_T(Y)=Y+TY & \mbox{``going upwards''}.
\end{array}
$$
Clearly $D_T(Y)\subseteq Y\subseteq U_T(Y)$.

\begin{lemma}\label{procedures-hs}
Let $Y\subseteq X$ be a half-space and $T\in L(X)$. If $Y$ is $T$-almost invariant then both $D_T(Y)$ and $U_T(Y)$ are half-spaces. Moreover, $\codim_YD_T(Y)=\codim_{U_T(Y)}Y=d_{Y,T}$.
\end{lemma}
\begin{proof}
The statement about $U_T(Y)$ follows immediately from the definition of an almost invariant subspace. Let's verify the statement about $D_T(Y)$. Clearly we only need to verify the ``moreover'' part.

Let $TY\subseteq Y+F$ where $F$ is such that $\dim F=d_{Y,T}$. By Lemma~\ref{char-min-dim}(i), we have $Y+F=Y\oplus F$. Let $P:Y\oplus F\to F$ be the projection along~$Y$. Then $D_T(Y)=\ker(PT|_Y)$. By Lemma~\ref{dim-codim} we get $\codim_YD_T(Y)=\dim({\rm Range}\ PT|_Y)=\dim F=d_{Y,T}$.
\end{proof}

The following lemma is the key statement of this chapter.

\begin{lemma}\label{key-lemma}
Suppose that $Y$ is a half-space in $X$ that is almost invariant under an operator $T\in L(X)$. If $d_{Y,T}>0$ and $d_{D^k_T(Y),T}\ge d_{Y,T}$ and $d_{U^k_T(Y),T}\ge d_{Y,T}$ for all $k\in\mathbb N$ then $d_{Y,T^m}\ge m$ for all $m\in\mathbb N$.
\end{lemma}

\begin{proof}
 
Denote for convenience $N=d_{Y,T}$. Fix $f_1^{0},\dots,f_N^{0}\in X$ such that $TY\subseteq Y\oplus F$ where $F=[f_i^{0}]_{i=1}^N$. In particular, $\dim F=d_{Y,T}$.
 
Suppose that $d_{D_T^k(Y),T}\ge N$ and $d_{U_T^k(Y),T}\ge N$ for all $k\in\mathbb N$. Denote $Y_0=Y$ and $Y_k=D^k_T(Y)$, $k\ge 1$. Since $Y_k=D_T(Y_{k-1})$ for all $k\ge 1$, it follows that $TY_k\subseteq Y_{k-1}$ as $k\ge 1$. 
 
We claim that for each $k\ge 1$ there exists an $N$-tuple $(f_1^{k},\dots,f_N^{k})$ in $Y_{k-1}$ such that 
\begin{enumerate}
\item $Y_k\oplus F_k=Y_{k-1}$ where $F_k=[f_i^{k}]_{i=1}^N$, and
\item if $P_k:Y_k\oplus F_k\to F_k$ is the projection along $Y_k$ then $(P_{k-1}T)f_i^{k}=f_i^{k-1}$ for all $i=1,\dots,N$ (if $k=1$ then we assume $P_0:Y\oplus F\to F$ is the projection along $Y$).
\end{enumerate}

Let $k=1$. By Lemma~\ref{char-min-dim}(iii), for each $i=1,\dots,N$, we can find $f_i^{1}\in Y$ such that $(P_0T)f_i^{1}=f_i^{0}$. Then (ii) is satisfied. Write $F_1=[f_i^{1}]_{i=1}^N$. Since $Y_1\cap F_1=\{0\}$ by definition of $Y_1$ and $\dim F_1=N=\codim_YY_1$ by Lemma~\ref{procedures-hs}, $Y_1\oplus F_1=Y$.

Suppose the claim is true for $k\ge 0$. Then $Y_k\oplus F_k=Y_{k-1}$. Since $TY_k\subseteq Y_{k-1}$ and $d_{Y_k,T}\ge N=\dim F_k$, we get $d_{Y_k,T}=N$. Then from Lemma~\ref{char-min-dim}(iii) for each $i=1,\dots,N$ there exists $f_i^{k+1}\in Y_k$ such that $(P_{k}T)f_i^{k+1}=f_i^{k}$, so that (ii) is satisfied for $k+1$. To show (i), write $F_{k+1}=[f_i^{k+1}]_{i=1}^N$ and observe: $Y_{k+1}\cap F_{k+1}=\{0\}$ by definition of $Y_{k+1}$ and $\dim F_{k+1}=N=\codim_{Y_k}Y_{k+1}$ by Lemma~\ref{procedures-hs}.

Observe that from condition (ii) of this claim we have: for each $k\ge 1$ there exists $y\in Y$ such that $T^kf_i^{k}=y+f_i^{0}$. That is, $f_i^{k}$ is a $k$-th ``preimage'' of $f_i^{0}$. It follows that any $f\in F$ has a $k$-th ``preimage'' in $Y_{k-1}$.

Denote $Z_0=Y$, $Z_k=U_T^k(Y)$, $k\ge 1$. That is, $Z_k=U_T(Z_{k-1})$. In particular, $TZ_{k-1}\subseteq Z_k$ for all $k\ge 0$. We claim that $Z_k=Y\oplus F\oplus TF\dots\oplus T^{k-1}F$. Indeed, for $k=0$ this is obvious. Suppose the claim is true for $k\ge 1$. Let's prove that $Z_{k+1}=Y\oplus F\oplus TF\dots\oplus T^{k}F$. We have $Z_{k+1}=U_T(Z_k)=Z_k+TZ_k=(Y\oplus F\oplus TF\dots\oplus T^{k-1}F)+(TY+TF+T^2F+\dots+T^kF)=(Y\oplus F\oplus TF\dots\oplus T^{k-1}F)+T^kF$ since $TY\subseteq Y\oplus F$. That is, $Z_{k+1}=Z_k+T^kF$. We only have to prove that this sum is direct. We have $\dim T^kF\le N$ since $\dim F=N$. On the other hand, $TZ_{k}\subseteq Z_{k+1}=Z_k+T^kF$. Since $d_{Z_k,T}\ge N$ for all $k\ge 0$, we get $\dim T^kF=N=d_{Z_k,T}$. By Lemma~\ref{char-min-dim}(i), the sum must be direct. 

Observe that in particular, this means that if $f\in F$ is non-zero then $T^kf\in Z_{k+1}\setminus Z_{k}$ ($k\ge 0$).

Let $u\in F$ be a non-zero vector, $m\in\mathbb N$ be arbitrary, and $k\in\{1,\dots,m\}$.  Put $u_k$ to be the $k$-th ``preimage'' of $u$, that is, such a vector in $Y_{k-1}$ that $T^ku_k=v_k+u$ for some $v_k\in Y$. Then $T^mu_k=T^{m-k}T^ku_k=T^{m-k}v_k+T^{m-k}u$. Since $v_k\in Y$, it follows that $T^{m-k}v_k\in Z_{m-k}$. Also since $u\ne 0$, we get $T^{m-k}u\in Z_{m-k+1}\setminus Z_{m-k}$. Since $Z_{m-k}\subseteq Z_{m-k+1}$ we have $T^mu_k\in Z_{m-k+1}\setminus Z_{m-k}$.
 
 This means that $T^mY$ contains $m$ vectors $\{T^m(u_k)\}_{k=1}^m$ such that 
 no non-zero linear combination of these vectors belongs to $Y$. By 
 Lemma~\ref{min-dim} we get: $d_{Y,T^m}\ge m$.
\end{proof}

As an immediate corollary we get:

\begin{theorem}\label{inv-existence}
Let $T\in L(X)$ be an operator and $\mathfrak A$ the norm closed algebra generated by $T$. If $\mathfrak A$ has an almost invariant half-space then $\mathfrak A$ has an invariant half-space.
\end{theorem}

\begin{proof}
If $d_{Y,T}=0$ then there is nothing to prove. Let $d_{Y,T}>0$. Since $\mathfrak A$ is norm closed, $\sup_{S\in\mathfrak A}d_{Y,S}<\infty$ by Theorem~\ref{uniform-bdd}. In particular, $\sup_{m\in\mathbb N}d_{Y,T^m}<\infty$. By Lemma~\ref{key-lemma} we obtain that either $d_{D_T^k(Y),T}<d_{Y,T}$ or $d_{U_T^k(Y),T}<d_{Y,T}$ for some $k\in\mathbb N$. 

Applying this finitely many times we get a half-space $Z$ such that $d_{Z,T}=0$. Since $Z$ is $T$-invariant, it is $\mathfrak A$-invariant.
\end{proof}

This theorem allows us to get an earlier promised example of a (not necessarily closed) algebra $\mathfrak A$ whose almost invariant half-spaces are different from those of $\overline{\mathfrak A}^{WOT}$ (and even from those of the norm closure of $\mathfrak A$). This example also shows that, unlike in case of invariant subspaces, there exists an operator whose almost invariant half-spaces are different from those of the norm closed algebra generated by this operator.

\begin{example} Let $D$ be a Donoghue operator on $\ell_2$. That is, $D$ is a backward shift with non-zero weights $(w_n)_{n=1}^\infty$ which satisfy conditions $(\abs{w_n})_{n=1}^\infty$ is monotone decreasing and is in $\ell_2$ (see, e.g.,~\cite[Section 4.4]{RR03} for the properties of Donoghue operators). Put $\mathfrak A=\{p(D)\colon p\mbox{ is a polynomial such that }p(0)=0\}$. It was proved in~\cite{APTT} that $D$ has an almost invariant half-space. Then $\mathfrak A$ has an almost invariant half-space. However all the invariant subspaces of $D$ are finite dimensional (see~\cite[Theorem 4.12]{RR03}) and therefore $D$ has no invariant half-spaces.
By Theorem~\ref{inv-existence}, $\overline{\mathfrak A}^{\,\norm{\cdot}}$ has no almost invariant half-spaces.
\end{example}

The following result is a generalization of Theorem~\ref{inv-existence}.

\begin{theorem}\label{inv-existence-ext} Let $\mathfrak A$ be a norm-closed algebra generated by a finite number of pairwise commuting operators. If $\mathfrak A$ has an almost invariant half-space then $\mathfrak A$ has an invariant half-space.
\end{theorem}

\begin{proof}
Let $\mathfrak A$ be generated by pairwise commuting operators $T_1,\dots,T_n$ and let $Y$ be an $\mathfrak A$-almost invariant half-space. We will prove that there exists a half-space that is invariant under $T_k$ for each $k=1,\dots, n$.

First, observe that if $T\in\mathfrak A$ then both $D_T(Y)$ and $U_T(Y)$ are $\mathfrak A$-almost invariant because $\codim_YD_T(Y)<\infty$ and $\codim_{U_T(Y)}Y<\infty$ by Lemma~\ref{procedures-hs}. Next, we claim that if $S\in\mathfrak A$ is such that $Y$ is $S$-invariant then $D_T(Y)$ and $U_T(Y)$ are again $S$-invariant. Indeed, let $y\in D_T(Y)$, then $T(Sy)=STy\in Y$ since $Ty\in Y$ and $Y$ is $S$-invariant. Hence $Sy\in D_T(Y)$, so that $D_T(Y)$ is $S$-invariant. Let $u+v\in U_T(Y)$, with $u\in Y$ and $v\in Ty$ for some $y\in Y$. Then $S(u+v)=Su+STy=Su+TSy\in U_T(Y)$ since $Su, Sy\in Y$, so that $U_T(Y)$ is $S$-invariant.

For $k=1,\dots,n$, denote by $\mathfrak A_k$ the norm closed algebra generated by $T_k$. Clearly $\mathfrak A_k\subseteq\mathfrak A$ and hence every $\mathfrak A$-almost invariant half-space is $\mathfrak A_k$-almost invariant for all $k=1,\dots,n$.

Apply a finite sequence of procedures $D_{T_1}$ and $U_{T_1}$ to $Y$ to obtain a $T_1$-invariant half-space $Y_1$, as in the proof of Theorem~\ref{inv-existence}. By the discussion above, $Y_1$ is $\mathfrak A$-almost invariant. Apply a finite sequence of procedures $D_{T_2}$ and $U_{T_2}$ to $Y_1$ to obtain a $T_2$-invariant half-space. Then $Y_2$ is $T_1$- and $T_2$-invariant and still $\mathfrak A$-almost invariant. Repeat this procedure $n-2$ more times to get an $\mathfrak A$-invariant half-space.
\end{proof}

\end{document}